\documentclass[11pt]{amsart} \usepackage{amsmath, amssymb}
\usepackage{amsfonts} \usepackage{graphicx} \usepackage{mathrsfs}
\usepackage{color}\usepackage{hyperref}
\usepackage[arrow,matrix,curve,cmtip,ps]{xy}

\usepackage{amsthm}

\allowdisplaybreaks

\newtheorem{theorem}{Theorem}[section]
\newtheorem{lemma}[theorem]{Lemma}
\newtheorem{proposition}[theorem]{Proposition}
\newtheorem{corollary}[theorem]{Corollary}

\newtheorem*{theorem*}{Theorem} \theoremstyle{remark}
\newtheorem{remark}[theorem]{Remark}
\newtheorem{definition}[theorem]{Definition}
\newtheorem{example}[theorem]{Example}

\newcommand{\fifi}{\mathbf{[11]}}
\newcommand{\fiin}{\mathbf{[1\infty]}}
\newcommand{\inin}{\mathbf{[\infty\infty]}}
\newcommand{\infi}{\mathbf{[\infty 1]}}

\numberwithin{equation}{section}

\newcommand{\smalltwobytwo}[4]{\left[\begin{smallmatrix}#1&#2\\#3&#4\end{smallmatrix}\right]}

\newcommand{\smalltwobyone}[2]{\left[\begin{smallmatrix}#1\\#2\end{smallmatrix}\right]}

\newcommand{\CCap}{\mathcal{C}_{API}}
\newcommand{\cstar}{\mbox{$C^*$}} \newcommand{\KKK}{\mathbb{K}}

\newcommand{\Z}{\mathbb{Z}} \newcommand{\ZZ}{\mathbb{Z}_{++}}
\newcommand{\ZZp}{\Z_{\pm}} \newcommand{\N}{\mathbb{N}}
 \newcommand{\C}{\mathbb{C}}
\newcommand{\T}{\mathbb{T}} \newcommand{\im}{\operatorname{im }}
\newcommand{\coker}{\operatorname{coker }}
 \newcommand{\ext}{\mathfrak{e}}
\newcommand{\Ksix}{K_\textbf{six}}

\newcommand{\gae}{\lower 2pt \hbox{$\, \buildrel {\scriptstyle >}\over
{\scriptstyle \sim}\,$}}

\newcommand{\lae}{\lower 2pt \hbox{$\, \buildrel {\scriptstyle <}\over
{\scriptstyle \sim}\,$}}

\newcommand{\cbstart}{} 
\newcommand{\cbend}{} 

\begin{document} \title{On the classification of nonsimple graph
$C^*$-algebras}

\author{S\o ren Eilers}

\author{Mark Tomforde} 

\address{Department of Mathematical Sciences\\University of
Copenhagen\\Universitetsparken 5\\2100 Copenhagen \O \\Denmark}
\email{eilers@math.ku.dk}

\address{Department of Mathematics \\ University of Houston \\
Houston, TX 77204-3008 \\USA} \email{tomforde@math.uh.edu}


\date{\today}



\begin{abstract} We prove that a graph $C^*$-algebra with exactly one
proper nontrivial ideal is classified up to stable isomorphism by its
associated six-term exact sequence in $K$-theory.  We prove that a
similar classification also holds for a graph $C^*$-algebra with a
largest proper ideal that is an AF-algebra. Our results are based on a
general method developed by the first named author with Restorff and
Ruiz. As a key step in the argument, we show how to produce stability
for certain full hereditary subalgebras associated to such graph
$C^*$-algebras.  We further prove that, except under trivial
circumstances, a unique proper nontrivial ideal in a graph
$C^*$-algebra is stable.
\end{abstract}

\maketitle

\section{Introduction}

The classification program for $C^*$-algebras has for the most part
progressed independently for the classes of infinite and finite
$C^*$-algebras.  Great strides have been made in this program for each
of these classes.  In the finite case, Elliott's Theorem classifies
all AF-algebras up to stable isomorphism by the ordered $K_0$-group.
In the infinite case, there are a number of results for purely
infinite $C^*$-algebras.  The Kirchberg-Phillips Theorem classifies
certain simple purely infinite $C^*$-algebras up to stable isomorphism
by the $K_0$-group together with the $K_1$-group.  For nonsimple
purely infinite $C^*$-algebras many partial results have been
obtained: R\o rdam has shown that certain purely infinite
$C^*$-algebras containing exactly one proper nontrivial ideal are
classified up to stable isomorphism by the associated six-term exact
sequence of $K$-groups \cite{Ror:ceccsteskt}, Restorff has shown that
nonsimple Cuntz-Krieger algebras satisfying Condition (II) are
classified up to stable isomorphism by their filtrated $K$-theory
\cite[Theorem~4.2]{gr:cckasi}, and Meyer and Nest have shown that
certain purely infinite $C^*$-algebras with a linear ideal lattice are
classified up to stable isomorphism by their filtrated $K$-theory
\cite[Theorem~4.14]{mn:cotpfkt}.  However, in all of these situations
the nonsimple $C^*$-algebras that are classified have the property
that they are either AF-algebras or purely infinite, and consequently
all of their ideals and quotients are of the same type.

Recently, the first named author with Restorff and Ruiz have provided
a framework for classifying nonsimple $C^*$-algebras that are not
necessarily AF-algebras or purely infinite $C^*$-algebras.  In
particular, these authors have shown in \cite{segrer:cecc} that
certain extensions of classifiable $C^*$-algebras may be classified up
to stable isomorphism by their associated six-term exact sequence in
\cbstart{}%
$K$-theory.  This has allowed for the classification of certain
\cbend{}%
nonsimple $C^*$-algebras in which there are ideals and quotients of
mixed type (some finite and some infinite).

In this paper we consider the classification of nonsimple graph
$C^*$-algebras.  Simple graph $C^*$-algebras are known to be either
AF-algebras or purely infinite algebras, and thus are classified by
their $K$-groups according to either Elliott's Theorem or the
Kirchberg-Phillips Theorem.  Therefore, we begin by considering
nonsimple graph $C^*$-algebras with exactly one proper nontrivial
ideal.  These $C^*$-algebras will be extensions of simple
$C^*$-algebras that are AF or purely infinite by other simple
$C^*$-algebras that are AF or purely infinite --- with mixing of the
types allowed.  These nonsimple graph $C^*$-algebras are similar to
the extensions considered in \cite{segrer:cecc}, however, the results
of \cite{segrer:cecc} do not apply directly.  Instead, we must do a
fair bit of work, using the techniques from the theory of graph
$C^*$-algebras, to show that the machinery of \cite{segrer:cecc} can
\cbstart{}%
be used to classify these extensions; it is verifying the requirement of
\textbf{fullness} that is most difficult in this
context.  Ultimately, however, we are able to show that a graph
$C^*$-algebra with exactly one proper nontrivial ideal is classified
up to stable isomorphism by the six-term exact sequence in $K$-theory
of the corresponding extension. 
\cbend{}%
Additionally, we are able to show
that a graph $C^*$-algebra with a largest proper ideal that is an
AF-algebra is also classified up to stable isomorphism by the six-term
exact sequence in $K$-theory of the corresponding extension.

It is also worthwhile to note that the extensions of graph
$C^*$-algebras classified in this paper constitute a very large class.
Every AF-algebra is stably isomorphic to a graph $C^*$-algebra, and
every Kirchberg algebra with free $K_1$-group is stably isomorphic to
a a graph $C^*$-algebra.  Thus the extensions we consider comprise a
wide variety of extensions of AF-algebras (respectively, purely
infinite algebras) by purely infinite algebras (respectively,
AF-algebras).

While there is little hope to generalize the methods in
\cite{segrer:cecc} to general (even finite) ideal lattices in a
context covering all graph $C^*$-algebras, the classifications we
obtain in this paper suggest that a complete classification of graph
$C^*$-algebras generalizing the Cuntz-Krieger case solved in
\cite{gr:cckasi} may be possible by other methods.  Such a result may
involve generalizing the work of Boyle and Huang to graph
$C^*$-algebras and mimicking the approach used by Restorff; or perhaps
it may be accomplished by generalizing Kirchberg's isomorphism theorem
to allow for subquotients which are $AF$-algebras, and then in this
special case (probably using the global vanishing of one connecting
map of $K$-theory) overcoming the difficulties of projective dimension
of the invariants exposed by Meyer and Nest. Neither of these
approaches seem within immediate reach, but both appear plausible to
be successful at some future stage.  It is also an open problem, and
possibly much less difficult, to establish isomorphism directly
between unital graph $C^*$-algebras by keeping track of the class of
\cbstart{}%
the unit in the $K_0$-groups.  We mention that the methods of
\cbend{}%
\cite{grer:rccconiII} do not seem to generalize to this setting.

This paper is organized as follows.  In \S\ref{not-conv-sec} we
establish notation and conventions for graph $C^*$-algebras and
extensions.  In \S\ref{prelims-sec} we derive a number of preliminary
results for graph $C^*$-algebras with the goal of applying the methods
of \cite{segrer:cecc}.  In \S\ref{class-sec} we use our results from
\S\ref{prelims-sec} and the results of \cite{segrer:cecc} to prove our
two main theorems: In Theorem~\ref{clas} we show that if $A$ is a
graph $C^*$-algebra with exactly one proper nontrivial ideal $I$, then
$A$ is classified up to stable isomorphism by the six-term exact
sequence in $K$-theory coming from the extension $0 \to I \to A \to
A/I \to 0$.  In Theorem~\ref{clasimp} we show that if $A$ is a graph
\cbstart{}%
$C^*$-algebra with a largest proper ideal $I$ that is an AF-algebra,
then $A$ is classified up to stable isomorphism by the six-term exact
\cbend{}%
sequence in $K$-theory coming from the extension $0 \to I \to A \to
A/I \to 0$.  In \S\ref{ex-sec} we consider a variety of examples, and
also use our results to classify the stable isomorphism classes of the
$C^*$-algebras of all graphs having exactly two vertices and
satisfying Condition~(K).  (Be aware that although these graphs have
only two vertices, the graphs are allowed to contain a finite or
countably infinite number of edges.)  We find that even for this small
collection of graphs, the associated $C^*$-algebras fall into a
variety of stable isomorphism classes, and there are quite a few cases
to consider.  We conclude in \S\ref{stab-ideals-sec} by proving that
if $A$ is a graph $C^*$-algebra that is not a nonunital AF-algebra,
and if $A$ contains a unique proper nontrivial ideal $I$, then $I$ is
stable.

\section{Notation and conventions} \label{not-conv-sec}

We establish some basic facts and notation for graph $C^*$-algebras
and extensions.

\subsection{Notation and conventions for graph $C^*$-algebra} A
(directed) graph $E=(E^0, E^1, r, s)$ consists of a countable set
$E^0$ of vertices, a countable set $E^1$ of edges, and maps $r,s: E^1
\rightarrow E^0$ identifying the range and source of each edge.  A
vertex $v \in E^0$ is called a \emph{sink} if $|s^{-1}(v)|=0$, and $v$
is called an \emph{infinite emitter} if $|s^{-1}(v)|=\infty$. A graph
$E$ is said to be \emph{row-finite} if it has no infinite emitters. If
$v$ is either a sink or an infinite emitter, then we call $v$ a
\emph{singular vertex}.  We write $E^0_\textnormal{sing}$ for the set
of singular vertices.  Vertices that are not singular vertices are
called \emph{regular vertices} and we write $E^0_\textnormal{reg}$ for
the set of regular vertices.  For any graph $E$, the \emph{vertex
matrix} is the $E^0 \times E^0$ matrix $A_E$ with $A_e(v,w) := | \{ e
\in E^1 : s(e)= v \text{ and } r(e)=w \} |$.  Note that the entries of
$A_E$ are elements of $\{0, 1, 2, \ldots \} \cup \{ \infty \}$.

If $E$ is a graph, a \emph{Cuntz-Krieger $E$-family} is a set of
mutually orthogonal projections $\{p_v : v \in E^0\}$ and a set of
partial isometries $\{s_e : e \in E^1\}$ with orthogonal ranges which
satisfy the \emph{Cuntz-Krieger relations}:
\begin{enumerate}
\item $s_e^* s_e = p_{r(e)}$ for every $e \in E^1$;
\item $s_e s_e^* \leq p_{s(e)}$ for every $e \in E^1$;
\item $p_v = \sum_{s(e)=v} s_e s_e^*$ for every $v \in E^0$ that is
not a singular vertex.
\end{enumerate} The \emph{graph algebra $C^*(E)$} is defined to be the
$C^*$-algebra generated by a universal Cuntz-Krieger $E$-family.

A \emph{path} in $E$ is a sequence of edges $\alpha = \alpha_1
\alpha_2 \ldots \alpha_n$ with $r(\alpha_i) = s(\alpha_{i+1})$ for $1
\leq i < n$, and we say that $\alpha$ has length $|\alpha| = n$.  We
let $E^n$ denote the set of all paths of length $n$, and we let $E^*
:= \bigcup_{n=0}^\infty E^n$ denote the set of finite paths in $G$.
Note that vertices are considered paths of length zero.  The maps
$r,s$ extend to $E^*$, and for $v,w \in E^0$ we write $v \geq w$ if
there exists a path $\alpha \in E^*$ with $s(\alpha)=v$ and $r(\alpha)
= w$.  Also for a path $\alpha := \alpha_1 \ldots \alpha_n$ we define
$s_\alpha := s_{\alpha_1} \ldots s_{\alpha_n}$, and for a vertex $v
\in E^0$ we let $s_v := p_v$.  It is a consequence of the
Cuntz-Krieger relations that $C^*(E) = \overline{\textrm{span}} \{
s_\alpha s_\beta^* : \alpha, \beta \in E^* \text{ and } r(\alpha) =
r(\beta)\}$.

We say that a path $\alpha := \alpha_1 \ldots \alpha_n$ of length $1$
or greater is a \emph{cycle} if $r(\alpha)=s(\alpha)$, and we call the
vertex $s(\alpha)=r(\alpha)$ the \emph{base point} of the cycle.  A
cycle is said to be \emph{simple} if $s(\alpha_i) \neq s(\alpha_1)$
for all $1 < i \leq n$.  The following is an important condition in
the theory of graph $C^*$-algebras.

$\text{ }$

\noindent \textbf{Condition~(K)}: No vertex in $E$ is the base point
of exactly one simple cycle; that is, every vertex is either the base
point of no cycles or at least two simple cycles.

$\text{ }$

For any graph $E$ a subset $H \subseteq E^0$ is \emph{hereditary} if
whenever $v, w \in E^0$ with $v \in H$ and $v \geq w$, then $w \in H$.
A hereditary subset $H$ is \emph{saturated} if whenever $v \in
E^0_\textnormal{reg}$ with $r(s^{-1}(v)) \subseteq H$, then $v \in H$.
For any saturated hereditary subset $H$, the \emph{breaking vertices}
corresponding to $H$ are the elements of the set $$B_H := \{ v \in E^0
: |s^{-1}(v)| = \infty \text{ and } 0 < |s^{-1}(v) \cap r^{-1}(E^0
\setminus H)| < \infty \}.$$ An \emph{admissible pair} $(H, S)$
consists of a saturated hereditary subset $H$ and a subset $S
\subseteq B_H$.  For a fixed graph $E$ we order the collection of
admissible pairs for $E$ by defining $(H,S) \leq (H',S')$ if and only
if $H \subseteq H'$ and $S \subseteq H' \cup S'$. For any admissible
pair $(H,S)$ we define
$$
I_{(H,S)} := \text{the ideal in $C^*(E)$ generated by $\{p_v : v \in
H\} \cup \{p_{v_0}^H : v_0 \in S\}$},
$$
where $p_{v_0}^H$ is the \emph{gap projection} defined by
$$
p_{v_0}^H := p_{v_0} - \sum_{{s(e) = v_0} \atop {r(e) \notin H}} s_e
s_e^*.
$$
Note that the definition of $B_H$ ensures that the sum on the right is
finite.

For any graph $E$ there is a canonical gauge action $\gamma :
\mathbb{T} \to \operatorname{Aut} C^*(E)$ with the property that for
any $z \in \mathbb{T}$ we have $\gamma_z (p_v) = p_v$ for all $v \in
E^0$ and $\gamma_z(s_e) = zs_e$ for all $e \in E^1$.  We say that an
ideal $I \triangleleft C^*(E)$ is \emph{gauge invariant} if
$\gamma_z(I) \subseteq I$ for all $z \in \mathbb{T}$.

There is a bijective correspondence between the lattice of admissible
pairs of $E$ and the lattice of gauge-invariant ideals of $C^*(E)$
given by $(H, S) \mapsto I_{(H,S)}$
\cite[Theorem~3.6]{bhrs:iccig}. When $E$ satisfies Condition~(K), all
ideals of $C^*(E)$ are gauge invariant \cite[Theorem~2.24]{mt:sgcg}
and the map $(H, S) \mapsto I_{(H,S)}$ is onto the lattice of ideals
of $C^*(E)$.  When $B_H = \emptyset$, we write $I_H$ in place of
\cbstart{}%
$I_{(H, \emptyset)}$ and observe that $I_H$ equals the ideal generated
\cbend{}%
by $\{ p_v : v \in H \}$.  Note that if $E$ is row-finite, then $B_H$
is empty for every saturated hereditary subset $H$.

\subsection{Notation and conventions for extensions} All ideals in
$C^*$-algebras will be considered to be closed two-sided ideals.  An
element $a$ of a $C^*$-algebra $A$ (respectively, a subset $S
\subseteq A$) is said to be \emph{full} if $a$ (respectively, $S$) is
not contained in any proper ideal of $A$.  A map $\phi : A \to B$ is
\emph{full} if $\im \phi$ is full in $B$.

If $A$ and $B$ are $C^*$-algebras, an \emph{extension} of $A$ by $B$
consists of a $C^*$-algebra $E$ and a short exact sequence
\[ \xymatrix{\ext: & 0 \ar[r] & B \ar[r]^\alpha & E \ar[r]^\beta & A
\ar[r] & 0.}
\] We say that the extension $\ext$ is \emph{essential} if
$\alpha(B)$ is an essential ideal of $E$, and we say that the extension $\ext$ is \emph{unital} if $E$ is a unital $C^*$-algebra.  For any extension there
exist unique $*$-homomorphisms $\eta_\ext : E \to \mathcal{M} (B)$ and
$\tau_\ext : A \to \mathcal{Q}(B) := \mathcal{M}(B)/B$ which make the
diagram
\[ \xymatrix{ 0 \ar[r] & B \ar[r]^\alpha \ar@{=}[d] & E \ar[r]^\beta
\ar[d]_{\eta_\ext} & A \ar[r] \ar[d]_{\tau_\ext} & 0 \\ 0 \ar[r] & B
\ar[r]^i & \mathcal{M}(B) \ar[r]^\pi & \mathcal{Q}(B) \ar[r] & 0}
\] commute.  The $*$-homomorphism $\tau_\ext$ is called the
\emph{Busby invariant} of the extension, and the extension is
essential if and only if $\tau_\ext$ is injective.  An extension
$\ext$ is \emph{full} if the associated Busby invariant $\tau_\ext$
has the property that $\tau_\ext(a)$ is full in $\mathcal{Q}(A)$ for
every $a \in A \setminus \{ 0 \}$.

For an extension $\ext$, we let $\Ksix (\ext)$ denote the cyclic
six-term exact sequence of $K$-groups
\[ \xymatrix{ {K_0(B)}\ar[r]&{K_0(E)}\ar[r]&{K_0(A)}\ar[d]\\
{K_1(A)}\ar[u]&{K_1(E)}\ar[l]&{K_1(B)}\ar[l]}
\] where $K_0(B)$, $K_0(E)$, and $K_0(A)$ are viewed as (pre-)ordered
groups.  Given two extensions
\cbstart{}%
\[ \xymatrix{\ext_1: & 0 \ar[r] & B_1 \ar[r]^{\alpha_1} & E_1 \ar[r]^{\beta_1}
& A_1 \ar[r] & 0 & \\ \ext_2: & 0 \ar[r] & B_2 \ar[r]^{\alpha_2} & E_2
\ar[r]^{\beta_2} & A_2 \ar[r] & 0 & }
\] we say $\Ksix(\ext_1)$ \emph{is isomorphic to} $\Ksix(\ext_2)$,
\cbend{}%
written $\Ksix(\ext_1) \cong \Ksix(\ext_2)$, if there exist
isomorphisms $\alpha$, $\beta$, $\gamma$, $\delta$, $\epsilon$, and
$\zeta$ making the following diagram commute
\[ \xymatrix{ K_0(B_1) \ar[rr] \ar[rd]_\alpha & & K_0(E_1) \ar[rr]
\ar[d]_\beta & & K_0(A_1) \ar[ddd] \ar[dl]^\gamma \\ &
{K_0(B_2)}\ar[r]&{K_0(E_2)}\ar[r]&{K_0(A_2)}\ar[d] & \\ &
{K_1(A_2)}\ar[u]&{K_1(E_2)}\ar[l]&{K_1(B_2)}\ar[l] & \\ K_1(A_1)
\ar[uuu] \ar[ru]^\zeta & & K_1(E_1) \ar[ll] \ar[u]^\epsilon & &
K_1(B_1) \ar[ll] \ar[ul]_\delta}
\] and where $\alpha$, $\beta$, and $\gamma$ are isomorphisms of
(pre-)ordered groups.

\section{Preliminary graph $C^*$-algebra results} \label{prelims-sec}

In this section we develop a few results for graph $C^*$-algebras in
order to apply the methods of \cite{segrer:cecc} in \S
\ref{class-sec}.  However, several of these results are interesting in
their own right.

\begin{lemma} \label{E-structure-lem} If $E$ is a graph such that
$C^*(E)$ contains a unique proper nontrivial ideal $I$, then the
following six conditions are satisfied:
\begin{enumerate}
\item $E$ satisfies Condition (K),
\item $E$ contains exactly three saturated hereditary subsets $\{
\emptyset, H, E^0 \}$,
\item $E$ contains no breaking vertices; i.e., $B_H = \emptyset$,
\item $I$ is a gauge invariant ideal and $I_H = I$,
\item If $X$ is a nonempty hereditary subset of $E$, then $X \cap H
\neq \emptyset$, and
\item $E$ has at most one sink, and if $v$ is a sink of $E$ then $v
\in H$.
\end{enumerate}
\end{lemma}

\begin{proof} Suppose that $E$ does not satisfy Condition~(K).  Then
by \cite[Proposition~1.17]{mt:sgcg} there exists a saturated
hereditary subset $H \subseteq E^0$ such that $E \setminus H$ contains
a cycle $\alpha = e_1 \ldots e_n$ with no exits.  The set $X = \{
s(e_i) \}_{i=1}^n$ is a hereditary subset of $E \setminus H$, and
$I_X$ is an ideal in $C^*(E \setminus H)$ Morita equivalent to $M_n
(C(\mathbb{T}))$ (see \cite[Proposition~3.4]{bhrs:iccig} and
\cite[Example~2.14]{rae:ga} for details).  Thus $I_X$, and hence
$C^*(E\setminus H)$, contains a countably infinite number of ideals.
Since $C^*(E \setminus H) \cong C^*(E) / I_{(H,B_H)}$
\cite[Proposition~3.4]{bhrs:iccig}, this implies that $C^*(E)$ has a
countably infinite number of ideals.  Hence if $C^*(E)$ has a finite
number of ideals, $E$ satisfies Condition~(K).

Because $E$ satisfies Condition~(K), it follows from
\cite[Theorem~3.5]{ddmt:ckegc} that the ideals of $C^*(E)$ are in
one-to-one correspondence with the pairs $(H, S)$ where $H$ is
saturated hereditary, and $S \subseteq B_H$ is a subset of the
breaking vertices of $H$.  Since $E$ contains a unique proper
nontrivial ideal, it follows that $E$ contains a unique saturated
hereditary subset $H$ not equal to $E^0$ or $\emptyset$, and that
there are no breaking vertices; i.e., $B_H = \emptyset$.  It must also
be the case that $I= I_H$.  Moreover, since $E$ satisfies
Condition~(K), \cite[Corollary~3.8]{bhrs:iccig} shows that all ideals
of $C^*(E)$ are gauge-invariant.

In addition, suppose $X$ is a hereditary subset with $X \cap H =
\emptyset$.  Since $H$ is hereditary, none of the vertices in $H$ can
reach $X$, and thus the saturation $\overline{X}$ contains no vertices
of $H$, and $\overline{X} \cap H = \emptyset$.  But then
$\overline{X}$ is a saturated hereditary subset of $E$ that does not
contain the vertices of $H$, and hence must be equal to $\emptyset$.
Thus if $X$ is a nonempty hereditary subset of $E$, then $X \cap H
\neq \emptyset$.

Finally, suppose $v$ is a sink of $E$.  Consider the hereditary subset
$X := \{ v \}$.  From the previous paragraph it follows that $X \cap H
\neq \emptyset$ and hence $v \in H$.  In addition, there cannot be a
second sink in $E$, for if $v'$ is a sink, then $X := \{ v \}$ and $Y
:= \{ v' \}$ are distinct hereditary sets.  Since $v$ cannot reach
$v'$, we see that $v$ is not in the saturation $\overline{Y}$.
Similarly, since $v'$ cannot reach $v$, we have that $v'$ is not in
the saturation $\overline{X}$.  Thus $\overline{X}$ and $\overline{Y}$
are distinct saturated hereditary subsets of $E$ that are proper and
nontrivial, which is a contradiction.  It follows that there is at
most one sink in $E$.
\end{proof}

\begin{remark} \label{cases-remark} According to
\cite[Lemma~1.3]{kdjhhsw:srga} and \cite[Corollary~3.5]{bhrs:iccig}
the $C^*$-algebras $I$ and $A/I$ are in this case also graph
algebras. As they are necessarily simple, they must be either
Kirchberg algebras or AF-algebras.  We will denote the four cases thus
occurring as follows

$ $

\cbstart{}%
\begin{center}
\begin{tabular}{|c|c|c|}\hline Case&$I$&$A/I$\\ \hline $\fifi$&AF&AF\\
$\fiin$&AF&Kirchberg\\ $\infi$&Kirchberg&AF\\
$\inin$&Kirchberg&Kirchberg\\ \hline
\end{tabular}
\end{center} \smallskip
\cbend{}%
\end{remark}

$ $

\begin{definition} Let $A$ be a $C^*$-algebra.  A proper ideal $I
\triangleleft A$ is a \emph{largest proper ideal} of $A$ if whenever
$J \triangleleft A$, then either $J \subseteq I$ or $J=A$.
\end{definition}

Observe that a largest proper ideal is always an essential ideal.
Also note that if $A$ is a $C^*$-algebra with a unique proper
nontrivial ideal $I$, then $I$ is a largest proper ideal; and if $A$
is a simple $C^*$-algebra then $\{ 0 \}$ is a largest proper ideal.

\begin{lemma} \label{largest-g-i} Let $E$ be a graph, and suppose that
$I$ is a largest proper ideal of $C^*(E)$.  Then $I$ is gauge
invariant and $I = I_{(H, B_H)}$ for some saturated hereditary subset
$H$ of $E^0$.  Furthermore, if $K$ is any saturated hereditary subset
of $E$, then either $K \subseteq H$ or $K = E^0$.
\end{lemma}

\begin{proof} Let $\gamma$ denote the canonical gauge action of $\T$
on $C^*(E)$.  For any $z \in \T$ we have that $\gamma_z(I)$ is a
proper ideal of $C^*(E)$.  Since $I$ is a largest proper ideal of
$C^*(E)$, it follows that $\gamma_z(I) \subseteq I$.  A similar
argument shows that $\gamma_{z^{-1}}(I) \subseteq I$.  Thus
$\gamma_z(I) = I$ and $I$ is gauge invariant.  It follows from
\cite[Theorem~3.6]{bhrs:iccig} that $I = I_{(H,S)}$ for some saturated
hereditary subset $H$ of $E^0$ and some subset $S \subseteq B_H$.
Because $I$ is a largest proper ideal, it follows that $S = B_H$, and
hence $I = I_{(H,B_H)}$.  Furthermore, if $K$ is a saturated
hereditary subset, then either $I_{(K,B_K)} \subseteq I_{(H,B_H)}$ or
$I_{(K,B_K)} = C^*(E)$.  Hence either $K \subseteq H$ or $K = E^0$.
\end{proof}

\begin{lemma} \label{cycle-exists-lem} Let $E$ be a graph and suppose
that $I$ is a largest proper ideal of $C^*(E)$ with the property that
$C^*(E) / I$ is purely infinite.  Then $I = I_{(H, B_H)}$ for some
saturated hereditary subset $H$ of $E^0$, and there exists a cycle
$\gamma$ in $E \setminus H$ and an edge $f \in E^1$ with $s(f) =
s(\gamma)$ and $r(f) \in H$.  Furthermore, if $x \in E^0$, then $x
\geq s(\gamma)$ if and only if $x \in E^0 \setminus H$.
\end{lemma}

\begin{proof} Lemma~\ref{largest-g-i} shows that $I = I_{(H, B_H)}$
for some saturated hereditary subset $H$ of $E^0$.  It follows from
\cite[Corollary~3.5]{bhrs:iccig} that $C^*(E) / I_{(H, B_H)} \cong
C^*(E \setminus H)$, where $E \setminus H$ is the subgraph of $E$ with
$(E \setminus H)^0 := E^0 \setminus H$ and $(E \setminus H)^1 := E^1
\setminus r^{-1}(H)$.  Since $C^*(E \setminus H)$ is purely infinite,
it follows from \cite[Corollary~2.14]{ddmt:cag} that $E \setminus H$
contains a cycle $\alpha$.  Define $K := \{ x \in E^0 : x \ngeq
s(\alpha) \}$.  Then $K$ is saturated hereditary, $H \subseteq K$, and
$K \neq E^0$.  Hence $I_{(H, B_H)} \subseteq I_{(K, B_K)} \neq
C^*(E)$, and the fact that $I_{(H, B_H)}$ is a largest proper ideal
implies that $I_{(H, B_H)} = I_{(K, B_K)}$ so that $H = K$.  Hence for
$x \in E^0$ we have $x \geq s(\alpha)$ if and only if $x \in E^0
\setminus H$.

Consider the set $J := \{ x \in E^0 : s(\alpha) \geq x \}$.  Then $J$
is a hereditary subset and we let $\overline{J}$ denote its
saturation.  Since $I_{(H,B_H)}$ is a largest proper ideal, it follows
that either $\overline{J} \subseteq H$ or $\overline{J} = E^0$.  Since
$s(\alpha) \in \overline{J} \setminus H$, we must have $\overline{J} =
E^0$.  Choose any element $w \in H$.  Since $w \in \overline{J}$ it
follows that there exists $v \in J$ with $w \geq v$.  But since $w
\geq v$ and $H$ is hereditary, it follows that $v \in H$.  Hence $v
\in J \cap H$, and there is a path from $s(\alpha)$ to a vertex in
$H$.  Choose a path $\mu = \mu_1 \mu_2 \ldots \mu_n$ with $s(\mu) =
s(\alpha)$, $r(\mu_{n-1}) \notin H$, and $r(\mu_n) \in H$.  Since
$r(\mu_{n-1}) \notin H$ the previous paragraph shows that there exists
a path $\nu$ with $s(\nu) = r(\mu_{n-1})$ and $r(\nu) = s(\alpha)$.
Let $\gamma := \nu \mu_1 \ldots \mu_{n-1}$ and let $f:= \mu_n$.  Then
$\gamma$ is a cycle in $E \setminus H$ and $f$ is an edge with $s(f) =
s(\gamma)$ and $r(f) \in H$.  Furthermore, since $s(\alpha)$ is a
vertex on the cycle $\gamma$, we see that for any $x \in E^0$ we have
$x \geq s(\gamma)$ if and only if $x \geq s(\alpha)$.  It follows from
the previous paragraph that if $x \in E^0$, then $x \geq s(\gamma)$ if
and only if $x \in E^0 \setminus H$.
\end{proof}

\begin{remark} Note that the conclusion of the above lemma does not
hold if $I$ is a maximal proper ideal that is not a largest proper
ideal.  For example, if $E$ is the graph $ $

$$
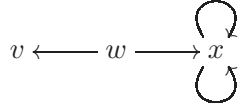
\xymatrix{ v & w \ar[l] \ar[r] & x  \ar@(ul,ur)[]  \ar@(dl,dr)[] }$$ 

$ $

\noindent and $H= \{v\}$ then $I := I_H$ is a maximal ideal that is
AF, and $C^*(E) / I_H \cong M_2 (\mathcal{O}_2)$ is purely infinite.
However, there is no edge from the base point of a cycle to $H = \{ v
\}$.
\end{remark}

\begin{definition} We say that two projections $p,q \in A$ are
\emph{equivalent}, written $p \sim q$, if there exists an element $v
\in A$ with $p=vv^*$ and $q=v^*v$.  We write $p \lesssim q$ to mean
that $p$ is equivalent to a subprojection of $q$; that is, there
exists $v \in A$ such that $p = vv^*$ and $v^*v \leq q$.  Note that $p
\lesssim q$ and $q \lesssim p$ does not imply that $p \sim q$ (unless
$A$ is finite).
\end{definition}

If $e \in G^1$ then we see that $p_{r(e)} = s_e^*s_e$ and $s_es_e^*
\leq p_{s(e)}$.  Therefore $p_{r(e)} \lesssim p_{s(e)}$.  More
generally we see that $v \geq w$ implies $p_w \lesssim p_v$.

\begin{lemma} \label{stab-equiv} Let $A$ be a $C^*$-algebra with an
increasing countable approximate unit $\{ p_n \}_{n=1}^\infty$
consisting of projections.  Then the following are equivalent.
\begin{enumerate}
\item[(i)] $A$ is stable.
\item[(ii)] For every projection $p \in A$ there exists a projection
$q \in A$ such that $p \sim q$ and $p \perp q$.
\item[(iii)] For all $n \in \N$ there exists $m > n$ such that $p_n
\lesssim p_m-p_n$
\end{enumerate}
\end{lemma}

\begin{proof} The equivalence of \textrm{(i)} and \textrm{(ii)} is
shown in \cite[Theorem~3.3]{jhmr:sc}.  The equivalence of
\textrm{(ii)} and \textrm{(iii)} is shown in
\cite[Lemma~2.1]{jh:piscgds}.
\end{proof}

\begin{lemma} \label{sum-equiv-projs} Let $A$ be a $C^*$-algebra.
Suppose $p_1, p_2, \ldots, p_n$ are mutually orthogonal projections in
$A$, and $q_1, q_2, \ldots, q_n$ are mutually orthogonal projections
in $A$ with $p_i \sim q_i$ for $1 \leq i \leq n$.  Then $\sum_{i=1}^n
p_i \sim \sum_{i=1}^n q_i$.
\end{lemma}

\begin{proof} Since $p_i \sim q_i$ there exists $v_i \in A$ such that
$v_i^*v_i = p_i$ and $v_iv_i^* = q_i$.  Thus for $i \neq j$ we have
$v_j^* v_i = v_j^*v_jv_j^* v_iv_i^*v_i = v_j^* q_j q_i v_i = 0$ and
$v_i v_j^* = v_iv_i^*v_i v_j^* v_j v_j^* = v_i p_i p_j v_j^* = 0$.
Hence $(\sum_{i=1}^n v_i)^* \sum_{i=1}^n v_i = \sum_{i=1}^n v_i^*v_i =
\sum_{i=1}^n p_i$ and $\sum_{i=1}^n v_i ( \sum_{i=1}^n v_i)^* =
\sum_{i=1}^n v_iv_i^* = \sum_{i=1}^n q_i$.  Thus $\sum_{i=1}^n p_i
\sim \sum_{i=1}^n q_i$.
\end{proof}

\begin{proposition}\label{getcorner} Let $E$ be a graph with no
breaking vertices, and suppose that $I$ is a largest proper ideal of
$C^*(E)$ and such that $C^*(E)/I$ is purely infinite and $I$ is AF.
Then there exists a projection $p \in C^*(E)$ such that $pC^*(E)p$ is
a full corner of $C^*(E)$ and $pIp$ is stable.
\end{proposition}

\begin{proof} Lemma~\ref{cycle-exists-lem} implies that $I =
I_{(H,B_H)}$ for some saturated hereditary subset $H$ of $E^0$, and
there exists a cycle $\gamma$ in $E \setminus H$ and an edge $f \in
E^1$ with $s(f) = s(\gamma)$ and $r(f) \in H$; and furthermore, if $x
\in E^0$, then $x \geq s(\gamma)$ if and only if $x \in E^0 \setminus
H$.  Since $E$ has no breaking vertices, we have that $B_H =
\emptyset$ so that $I_{(H,B_H)}$ is the ideal generated by $\{p_v : v
\in H \}$ and we may write $I_{(H,B_H)}$ as $I_H$.

Let $v = s(f) = s(\gamma)$ and let $w = r(f)$.  Define $p := p_v +
p_w$.  Suppose $J \triangleleft C^*(E)$ and $pC^*(E)p \subseteq J$.
Since $v \notin H$ we see that $p_v \notin I$ and hence $p_v \in
pC^*(E)p \setminus I \subseteq J \setminus I$.  Thus $J \nsubseteq I$
and the fact that $I$ is a largest proper ideal implies that $J =
C^*(E)$.  Hence $pC^*(E)p$ is a full corner of $C^*(E)$.

In addition, since there are no breaking vertices
\begin{align*} pIp &= pI_Hp \\ &= p \left(
\overline{\operatorname{span}} \{ s_\alpha s_\beta^* : r(\alpha) =
r(\beta) \in H \} \right)p \\ &= \overline{\operatorname{span}} \{
ps_\alpha s_\beta^*p : r(\alpha) = r(\beta) \in H \} \\ &=
\overline{\operatorname{span}} \{ s_\alpha s_\beta^* : r(\alpha) =
r(\beta) \in H \text{ and } s(\alpha), s(\beta) \in \{v, w \} \}.
\end{align*} Let $S := \{ \alpha \in E^* : s(\alpha) = v \text{ and }
r(\alpha) = w \}$.  Since $S$ is a countable set we may list the
elements of $S$ and write $S = \{ \alpha_1, \alpha_2, \alpha_3, \ldots
\}$.  Define $p_0 := p_w$ and $p_n := p_w + \sum_{k=1}^n s_{\alpha_k}
s_{\alpha_k}^*$ for $n \in \N$.

We will show that for $\mu, \nu \in S$ we have
\begin{equation} \label{S-eqn-displayed} s_\mu^* s_\nu
:= \begin{cases} p_{r(\mu)} & \text{ if $\mu = \nu$} \\ 0 & \text{
otherwise.} \end{cases}
\end{equation} First suppose that $s_\mu^* s_\nu \neq 0$.  Then one of
$\mu$ and $\nu$ must extend the other.  Suppose $\mu$ extends $\nu$.
Then $\mu = \nu \lambda$ for some $\lambda \in E^*$.  Thus $s(\lambda)
= r(\nu) = w$ and $r(\lambda) = r(\mu) = w$.  However, $I_H$ is an
AF-algebra, and $C^*(E_H)$ is strongly Morita equivalent to $I_H$
\cite[Proposition~3.4]{bhrs:iccig}, so $C^*(E_H)$ is an AF-algebra.
Thus $E_H$ contains no cycles.  Since $\lambda$ is a path in $E_H$
with $s(\lambda) = r(\lambda) = w$, and since $E_H$ contains no
cycles, we may conclude that $\lambda = w$.  Thus $\mu = \nu$.  A
similar argument works when $\nu$ extends $\mu$.  Hence the equation
in \eqref{S-eqn-displayed} holds.  It follows that the elements of the
set $\{ s_\alpha s_\alpha^* : \alpha \in S \} \cup \{p_w \}$ are
mutually orthogonal projections, and hence $\{ p_n \}_{n=0}^\infty$ is
an sequence of increasing projections.

Next we shall show that $\{ p_n \}_{n=0}^\infty$ is an approximate
unit for $pIp$.  Given $s_\alpha s_\beta^*$ with $r(\alpha) = r(\beta)
\in H$ and $s(\alpha), s(\beta) \in \{v, w \}$, we consider two cases.

\noindent \textsc{Case I:} $s(\alpha) = w$.  Then for any $\alpha_k
\in S$ we see that $(s_{\alpha_k} s_{\alpha_k}^*) s_\alpha s_\beta^* =
s_{\alpha_k} s_{\alpha_k}^* p_w s_\alpha s_\beta^* = 0$.  In addition,
$p_w (s_\alpha s_\beta^*) = s_\alpha s_\beta^*$.  Thus $\lim_{n \to
\infty} p_n s_\alpha s_\beta^* = s_\alpha s_\beta^*$.

\noindent \textsc{Case II:} $s(\alpha) = v$.  Then $\alpha = \alpha_j
\lambda$ for some $\alpha_j \in S$ and some $\lambda \in E_H^*$ with
$s(\lambda) = w$.  We have $p_w (s_\alpha s_\beta) = 0$, and also
\eqref{S-eqn-displayed} implies that $$(s_{\alpha_k} s_{\alpha_k}^*)
s_\alpha s_\beta^* = s_{\alpha_k} s_{\alpha_k}^* s_{\alpha_j}
s_\lambda s_\beta^* = \begin{cases} s_\alpha s_\beta^* & \text{ $k =
j$} \\ 0 & \text{ $k \neq j$.} \end{cases}$$ Thus $\lim_{n \to \infty}
p_n s_\alpha s_\beta^* = s_\alpha s_\beta^*$.

The above two cases imply that $\lim_{n \to \infty} p_n x = x$ for any
$x \in \operatorname{span} \{ s_\alpha s_\beta^* : r(\alpha) =
r(\beta) \in H \text{ and } s(\alpha), s(\beta) \in \{v, w \} \}$.
Furthermore, an $\epsilon / 3$-argument shows that $\lim_{n \to
\infty} p_n x = x$ for any $x \in pI_Hp =
\overline{\operatorname{span}} \{ s_\alpha s_\beta^* : r(\alpha) =
r(\beta) \in H \text{ and } s(\alpha), s(\beta) \in \{v, w \} \}$.  A
similar argument shows that $\lim_{n \to \infty} x p_n = x$ for any $x
\in pI_Hp$.  Thus $\{ p_n \}_{n=1}^\infty$ is an approximate unit for
$pI_Hp$.

We shall now show that $pI_H p$ is stable.  For each $n \in \N$
define $$\lambda^n := \underbrace{ \gamma \gamma \ldots
\gamma}_\text{$n$ times} f.$$ For any $k, n \in \N$ we
have $$s_{\lambda^n} s_{\lambda^n}^* \sim s_{\lambda^n}^*
s_{\lambda^n} = p_{r(\lambda^n)} = p_w = s_{\alpha_k}^* s_{\alpha_k}
\sim s_{\alpha_k} s_{\alpha_k}^*.$$

For any $n \in \N$ choose $q$ large enough that $|\lambda^q| \geq
|\alpha_k|$ for all $1 \leq k \leq n$.  Then for all $j \in \N$ we see
that $\lambda^{q+j} \in S$ and $\lambda^{q+j} \neq \alpha_k$ for all
$1 \leq k \leq n$.  Thus for any $1 \leq k \leq n$ we
have $$s_{\alpha_k} s_{\alpha_k}^* \sim s_{\alpha_k}^* s_{\alpha_k} =
p_{r(\alpha_k)} = p_w = p_{r(\lambda^{q+k})} = s_{\lambda^{q+k}}^*
s_{\lambda^{q+k}} \sim s_{\lambda^{q+k}} s_{\lambda^{q+k}}^*$$
and $$p_w = s_{\lambda^q}^* s_{\lambda^q} \sim s_{\lambda^q}
s_{\lambda^q}^*.$$ It follows from Lemma~\ref{sum-equiv-projs}
that $$p_n = p_w + \sum_{k=1}^n s_{\alpha_k} s_{\alpha_k}^* \lesssim
\sum_{k=0}^n s_{\lambda^{q+k}} s_{\lambda^{q+k}}^* \lesssim p_m -
p_n$$ where $m$ is chosen large enough that $\lambda^{q+k} \in \{
\alpha_1, \alpha_2, \ldots, \alpha_m \}$ for all $0 \leq k \leq n$.
Lemma~\ref{stab-equiv} shows that $pI_Hp$ is stable.
\end{proof}

\section{Classification} \label{class-sec}

In this section we state and prove our main results.  We apply the
methods of \cite{segrer:cecc} to classify certain extensions of graph
$C^*$-algebras in terms of their six-term exact sequences of
$K$-groups.  To do this we will need to discuss classes of
$C^*$-algebras satisfying various properties.  We give definitions of
these properties here, and obtain a lemma that is a consequence of
\cite[Theorem~3.10]{segrer:cecc}.

\begin{definition}[see Definition~3.2 of
\cite{segrer:cecc}] \label{Class-4-props-def} We will be interested in
classes $\mathcal{C}$ of separable nuclear unital simple $C
^*$-algebras in the bootstrap category $\mathcal{N}$ satisfying the
following properties:
\begin{itemize}
\item[(I)] Any element of $\mathcal{C}$ is either purely infinite or
stably finite.
\item[(II)] $\mathcal{C}$ is closed under tensoring with
$\textrm{M}_n$, where $\textrm{M}_n$ is the $C^*$-algebra of $n$ by
$n$ matrices over $\mathcal{C}$.
\item[(III)] If $A$ is in $\mathcal{C}$, then any unital hereditary
$C^*$-subalgebra of $A$ is in $\mathcal{C}$.
\item[(IV)] For all $A$ and $B$ in $\mathcal{C}$ and for all $x$ in
$KK(A, B)$ which induce an isomorphism from $(K^+_*(A), [1_A])$ to
$(K^+_* (A), [1_B ])$, there exists a $*$-isomorphism $\alpha : A \to
B$ such that $KK (\alpha) = x$.
\end{itemize}
\end{definition}

\begin{definition} If $B$ is a separable stable $C^*$-algebra, then we
say that $B$ has the \emph{corona factorization property} if every
full projection in $\mathcal{M} (B)$ is Murray-von Neumann equivalent
to $1_{\mathcal{M}(B)}$.
\end{definition}

\begin{lemma}[Cf.~Theorem~3.10 of \cite{segrer:cecc}] \label{ERR-lem}
Let $\mathcal{C}_I$ and $\mathcal{C}_Q$ be classes of unital nuclear
separable simple $C^*$-algebras in the bootstrap category
$\mathcal{N}$ satisfying the properties of
Definition~\ref{Class-4-props-def}.  Let $A_1$ and $A_2$ be in
$\mathcal{C}_Q$ and let $B_1$ and $B_2$ be in $\mathcal{C}_I$ with
$B_1 \otimes \KKK$ and $B_2 \otimes \KKK$ satisfying the corona
factorization property.  Let \cbstart{}%
\[ \xymatrix{ {\ext_1:}&{0}\ar[r]&{B_1 \otimes
\KKK}\ar[r]&{E_1}\ar[r]&{A_1}\ar[r]&{0} & }
\]
\[ \xymatrix{ {\ext_2:}&{0}\ar[r]&{B_2 \otimes
\KKK}\ar[r]&{E_2}\ar[r]&{A_2}\ar[r]&{0} & }
\] be \textbf{essential} and \textbf{unital} extensions.  If $\Ksix
\cbend{}%
(\ext_1) \cong \Ksix (\ext_2)$, then $E_1 \otimes \KKK \cong E_2
\otimes \KKK$.
\end{lemma}

\begin{proof} Tensoring the extension $\ext_1$ by $\KKK$ we obtain a
short exact sequence $\ext_1'$ and and vertical maps
\[ \xymatrix{ {\ext_1:}&{0}\ar[r]&{B_1 \otimes
\KKK}\ar[r]\ar@{^{(}->}[d]&{E_1}\ar[r]\ar@{^{(}->}[d]&{A_1}\ar@{^{(}->}[d]\ar[r]
&{0}\\ {\ext_1':}&{0}\ar[r]&{(B_1 \otimes \KKK) \otimes
\KKK}\ar[r]&{E_1 \otimes \KKK}\ar[r]&{A_1 \otimes \KKK}\ar[r]&{0}}
\] from $\ext_1$ into $\ext'_1$ that are full inclusions.  These full
inclusions induce isomorphisms of $K$-groups and hence we have that
$\Ksix (\ext_1) \cong \Ksix (\ext_1')$.  In addition, since $\ext_1$
is essential, $B_1 \otimes \KKK$ is an essential ideal in $E_1$, and
the Rieffel correspondence between the strongly Morita equivalent
$C^*$-algebras $E_1$ and $E_1 \otimes \KKK$ implies that $(B_1 \otimes
\KKK) \otimes \KKK$ is an essential ideal in $E_1 \otimes \KKK$, so
that $\ext_1'$ is an essential extension.  Furthermore, since $B_1
\otimes \KKK$ is stable and $\ext_1$ is essential and full \cite[Proposition~1.5]{segrer:cecc}, it follows
from \cite[Proposition~1.6]{segrer:cecc} that $\ext_1'$ is full.
Moreover, since $\KKK \otimes \KKK \cong \KKK$, we may rewrite
$\ext_1'$ as
\[ \xymatrix{ {\ext_1':}&{0}\ar[r]&{B_1 \otimes \KKK}\ar[r]&{E_1
\otimes \KKK}\ar[r]&{A_1 \otimes \KKK}\ar[r]&{0}.}
\] By a similar argument, there is an essential and full extension
\[ \xymatrix{ {\ext_2':}&{0}\ar[r]&{B_2 \otimes \KKK}\ar[r]&{E_2
\otimes \KKK}\ar[r]&{A_2 \otimes \KKK}\ar[r]&{0}}
\] such that $\Ksix (\ext_2') \cong \Ksix (\ext_2)$.  Thus $\Ksix
(\ext_1') \cong \Ksix (\ext_2')$, and \cite[Theorem~3.10]{segrer:cecc}
implies that $E_1 \otimes \KKK \cong E_2 \otimes \KKK$.
\end{proof}

\begin{lemma} \label{full-inclusions-lem} Let $A$ be a $C^*$-algebra
and let $I$ be a largest proper ideal of $A$.  If $p \in A$ is a full
projection, then the inclusion map $pIp \hookrightarrow I$ and the
inclusion map $pAp/pIp \hookrightarrow A/I$ are both full inclusions.
\end{lemma}

\begin{proof} Since $p$ is a full projection, we see that $A$ is
Morita equivalent to $pAp$ and the Rieffel correspondence between
ideals takes the form $J \mapsto pJp$.  If $J$ is an ideal of $I$ with
$pIp \subseteq J$, then by compressing by $p$ we obtain $pIp \subseteq
pJp$.  Since the Rieffel correspondence is a bijection, this implies
that $I \subseteq J$, and because $J$ is an ideal contained in $I$, we
get that $I = J$.  Hence $pIp \hookrightarrow I$ is a full inclusion.
Furthermore, because $I$ is a largest proper ideal of $A$, we know
that $A/I$ is simple and thus $pAp/pIp \hookrightarrow A/I$ is a full
inclusion.
\end{proof}

\begin{theorem} \label{clas} If $A$ is a graph $C^*$-algebra with
exactly one proper nontrivial ideal $I$, then $A$ classified up to
stable isomorphism by the six-term exact sequence
\[ \xymatrix{ {K_0(I)}\ar[r]&{K_0(A)}\ar[r]&{K_0(A/I)}\ar[d]\\
{K_1(A/I)}\ar[u]&{K_1(A)}\ar[l]&{K_1(I)}\ar[l]}
\] with all $K_0$-groups considered as ordered groups. In other words,
if $A$ is a graph $C^*$-algebras with precisely one proper nontrivial
ideal $I$, if $A'$ is a graph $C^*$-algebras with precisely one proper
nontrivial ideal $I'$, and if
\[ \xymatrix{ {\ext_1:}&{0}\ar[r]& \, I \, \ar[r]&{\, A \,}\ar[r]&{\,
A/I \, }\ar[r]&{0} & }
\]
\[ \xymatrix{
{\ext_2:}&{0}\ar[r]&I'\ar[r]&{A'}\ar[r]&{A'/I'}\ar[r]&{0} & }
\] are the associated extensions, then $A \otimes \KKK \cong A'
\otimes \KKK$ if and only if $\Ksix (\ext_1) \cong \Ksix (\ext_2)$.

Moreover, in cases $\fiin$, $\infi$, and $\inin$, the order structure
on $K_0(A)$ may be removed from the invariant leaving it still
complete. And in case $\fifi$, the ordered group $K_0(A)$ is a
complete invariant in its own right.
\end{theorem}

\begin{proof} It is straightforward to show that $A \otimes \KKK \cong
A' \otimes \KKK$ implies that $\Ksix (\ext_1) \cong \Ksix (\ext_2)$.
Thus we need only establish the converse.  To do this, we begin by
assuming that $\Ksix (\ext_1) \cong \Ksix (\ext_2)$.

We define $\CCap$ as the union of the class of unital simple and
separable $AF$-algebras and the class of simple, nuclear, unital, and
separable purely infinite \cstar-algebras in the bootstrap category.
The category $\CCap$ meets all the requirements in the list in
Definition~\ref{Class-4-props-def}: We clearly have that each algebra
in $\CCap$ is either purely infinite or stably finite, and that
$\CCap$ is closed under passing to matrices and unital hereditary
subalgebras. We also need to prove that the Elliott invariant is
complete for $\CCap$, and this follows by the classification results
of Elliott \cite[Theorem~4.3]{gae:cilssfa} and Kirchberg-Phillips (see
\cite[Theorem~C]{kir:cpickt} and \cite[\S 4.2]{phi:ctnpisc}) after
noting that the classes are obviously distinguishable by the nature of
the positive cone in $K_0$.  Finally, as recorded in
\cite[Theorem~3.9]{segrer:cecc}, the stabilizations of the
$C^*$-algebras in $\CCap$ all have the corona factorization property
according to \cite[Theorem 5.2]{dkpwn:cfpaue} and \cite[Proposition
2.1]{pwn:cfp}.

 It follows from \cite[Lemma~1.3]{kdjhhsw:srga} and
\cite[Corollary~3.5]{bhrs:iccig} that $I$ and $A/I$ are simple graph
$C^*$-algebras and thus each of $I$ and $A/I$ is either an AF-algebra
or a purely infinite algebra \cite[Remark~2.16]{ddmt:cag}.  Similarly
for $I'$ and $A'/I'$.  Since $\Ksix (\ext_1) \cong \Ksix (\ext_2)$, we
see that $K_0(I) \cong K_0(I')$ and $K_0(A/I) \cong K_0(A'/I')$ as
ordered groups.  By considering the positive cone in these groups, we
may conclude that $I$ and $I'$ are either both purely infinite or both
AF-algebras, and also $A/I$ and $A'/I'$ are either both purely
infinite or both AF-algebras.  Thus $A$ and $A'$ both fall into one of
the four cases described in Remark~\ref{cases-remark}.

\noindent\textbf{Cases $\inin$, $\infi$}\\ Write $A = C^*(E)$ for some
graph $E$.  Since $I$ is a largest proper ideal,
Lemma~\ref{largest-g-i} implies that $I = I_{(H,B_H)}$ for some
saturated hereditary subset $H \subsetneq E^0$.  If we let $v \in E^0
\setminus H$, and define $p :=p_v$, then $p \notin I$.  Since $I$ is a
largest proper ideal in $A$, this implies that the projection $p$ is
full.  Thus we obtain a full hereditary subalgebra $pAp$, and as noted
in Lemma~\ref{full-inclusions-lem} we have that all vertical maps in
\[ \xymatrix{
{\ext'_1:}&{0}\ar[r]&{pIp}\ar[r]\ar@{^{(}->}[d]&{pAp}\ar[r]\ar@{^{(}->}[d]&{pAp/pIp}\ar@{^{(}->}[d]\ar[r]
&{0}\\ {\ext_1 :}&{0}\ar[r]&{I}\ar[r]&{A}\ar[r]&{A/I}\ar[r]&{0}}
\] are full inclusions.  It follows that all of the above maps induce
isomorphisms on the $K$-groups and $\Ksix (\ext_1) \cong \Ksix
(\ext'_1)$.

In addition, since $pIp$ is nonunital and purely infinite, the ideal
$pIp$ is stable (by Zhang's dichotomy) and we may write $pIp \cong B_1
\otimes \KKK$ for a suitably chosen $B_1 \in \CCap$.  We now let $E_1
:= pAp$ and $A_1 := pAp / pIp$.  With this notation, $\ext'_1$ takes
the form
\[ \xymatrix{ {\ext_1':}&{0}\ar[r]&{B_1 \otimes
\KKK}\ar[r]&{E_1}\ar[r]&{A_1}\ar[r]&{0} }
\] with $B_1$ and $A_1$ unital $C^*$-algebras in $\CCap$.
\cbstart{}%
Furthermore, $\ext_1'$ is an essential extension because the ideal $I$
is a largest proper ideal in $A$, and thus also the ideal $pIp \cong
B_1 \otimes \KKK$ is a largest proper ideal in $A_1 \otimes \KKK$,
which implies that $pIp \cong B_1 \otimes \KKK$ is an essential ideal.
\cbend{}%

By a similar argument, we may find an extension
\[ \xymatrix{ {\ext_2':}&{0}\ar[r]&{B_2 \otimes
\KKK}\ar[r]&{E_2}\ar[r]&A_2 \ar[r]&{0}}
\] with $E_2 := qA'q$ for a full projection $q \in A'$, the
$C^*$-algebras $B_2$ and $A_2$ in $\CCap$ with $B_2 \otimes \KKK$
satisfying the corona factorization property, and $\ext_2'$ an
essential and full extension with $\Ksix (\ext_2') \cong \Ksix
(\ext_2)$.  It follows from Lemma~\ref{ERR-lem} that $E_1 \otimes \KKK
\cong E_2 \otimes \KKK$, or equivalently, that $pAp \otimes \KKK \cong
qA'q \otimes \KKK$.  Furthermore, because $pAp$ is a full corner of
$A$, and $qAq$ is a full corner of $A'$, we obtain that $pAp \otimes
\KKK \cong A \otimes \KKK$ and $qA'q \otimes \KKK \cong A' \otimes
\KKK$.  It follows that $A \otimes \KKK \cong A' \otimes \KKK$.  We
also observe that in this case the order structure on $K_0(A)$ is a
redundant part of the invariant.

\noindent\textbf{Case $\fiin$}\\ As seen in Lemma
\ref{E-structure-lem}, we may write $A=C^*(E)$ where $E$ has no
breaking vertices.  By Proposition~\ref{getcorner} there exists a
projection $p \in A$ such that $pAp$ is a full corner inside $A$, and
$pIp$ is stable.  Moreover, since $I$ is an AF-algebra by hypothesis
and $pIp$ is a hereditary subalgebra of $I$, it follows from
\cite[Theorem~3.1]{gae:admic} that $pIp$ is an AF-algebra.  Hence we
may choose a unital AF-algebra $B_1$ with $pIp \cong B_1 \otimes
\KKK$. The extension \cbstart{}%
\[ \xymatrix{ {\ext_1':}&{0}\ar[r]&{B_2 \otimes \KKK
}\ar[r]&{pAp}\ar[r]&{pAp/pIp}\ar[r] &{0}}
\] is essential.  In addition, an argument as in Cases $\inin$,
$\infi$ shows that $\Ksix (\ext_1') \cong \Ksix (\ext_1)$. 
\cbend{}%
We may
perform a similar argument for $A'$, and arguing as in Cases $\inin$,
$\infi$ and applying Lemma~\ref{ERR-lem} we obtain that $A \otimes
\KKK \cong A' \otimes \KKK$.  Again, the order structure on $K_0(A)$
is a redundant part of the invariant.

\noindent\textbf{Case $\fifi$}\\ Since $I$ and $A/I$ are AF-algebras,
it follows from a result of Brown that $A$ is an AF-algebra
\cite{bro:eaplp} (or see \cite[\S 9.9]{ege:dca} for a detailed proof).
Similarly, $A'$ is an AF-algebra.  It follows from Elliott's Theorem
that $K_0(A)$ order isomorphic to $K_0(A')$ implies that $A \otimes
\KKK \cong A' \otimes \KKK$.  Moreover, in this case the ordered group
$K_0(A)$ is a complete invariant.
\end{proof}

\begin{remark}
Joint work in progress by Ruiz and the first named author provides
information about the necessity of using order on the $K_0$-groups of
$\Ksix(-)$. There are examples of pairs of non-isomorphic stable AF-algebras $A$ and $A'$ with
exactly one ideal $I$ and $I'$ such that $\Ksix(\ext)\simeq
\Ksix(\ext')$ with group isomorphisms which are positive at $K_0(I)$
and $K_0(A/I)$ but not at $K_0(A)$. By \cite[Corollary~4.8]{kst:rafagaelua}, $A$ and $A'$ may be
realized as graph $C^*$-algebras.
In the other cases, one may prove that any isomorphism
between $\Ksix(\ext)$ and $\Ksix(\ext')$  will automatically be
positive at $K_0(A)$ if it is positive at $K_0(I)$ and $K_0(A/I)$. Thus it is possible that any isomorphism of our reduced invariant lifts to a $*$-isomorphism in the $\inin$, $\infi$, and $\fiin$ cases, but this has only been established in the $\inin$ case, cf.\  \cite{eilres}.
\end{remark}

Our proof of Theorem~\ref{clas} can be modified slightly to
give us an additional result.

\begin{theorem}\label{clasimp} If $A$ is a the $C^*$-algebra of a
graph satisfying Condition~(K), and if $A$ has a largest proper ideal
$I$ such that $I$ is an AF-algebra, then $A$ is classified up to
stable isomorphism by the six-term exact sequence
\[ \xymatrix{ {K_0(I)}\ar[r]&{K_0(A)}\ar[r]&{K_0(A/I)}\ar[d]\\
{K_1(A/I)}\ar[u]&{K_1(A)}\ar[l]&{K_1(I)}\ar[l]}
\] with $K_0(I)$ considered as an ordered group.

In other words, if $A$ is the $C^*$-algebra of a graph satisfying
Condition~(K) with a largest proper ideal $I$ that is an AF-algebra,
if $A'$ is the $C^*$-algebra of a graph satisfying Condition~(K) with
a largest proper ideal $I'$ that is an AF-algebra, and if
\[ \xymatrix{ {\ext_1:}&{0}\ar[r]& \, I \, \ar[r]&{\, A \,}\ar[r]&{\,
A/I \, }\ar[r]&{0} & }
\]
\[ \xymatrix{
{\ext_2:}&{0}\ar[r]&I'\ar[r]&{A'}\ar[r]&{A'/I'}\ar[r]&{0} & }
\] are the associated extensions, then $A \otimes \KKK \cong A'
\otimes \KKK$ if and only if $\Ksix (\ext_1) \cong \Ksix (\ext_2)$.
\end{theorem}

\begin{proof} To begin, using the desingularization of \cite{ddmt:cag}
\cbstart{}%
we may find a row-finite graph $F$ such that $C^*(F)$ is stably
\cbend{}%
isomorphic to $A$.  Since $C^*(F)$ is Morita equivalent to $A$, the
$C^*$-algebra $C^*(F)$ has a largest proper ideal that is an
AF-algebra, and the associated six-term exact sequence of $K$-groups
is isomorphic to $\Ksix (\ext_1)$.  Hence we may replace $A$ by
$C^*(F)$ for the purposes of the proof.  Likewise for $A'$.  Thus we
may, without loss of generality, assume that $A$ and $A'$ are
$C^*$-algebras of row-finite graphs, and in particular that $A$ and
$A'$ are $C^*$-algebras of graphs with no breaking vertices.  To
obtain the result, we simply argue as in Case $\fiin$ of the proof of
Theorem~\ref{clas}, using \cite[Theorem~3.13]{segrer:cecc} in place of
\cite[Theorem~3.10]{segrer:cecc}, and noting that
Proposition~\ref{getcorner} applies since the graphs have no breaking
vertices.
\end{proof}

\section{Examples} \label{ex-sec}

To illustrate our methods we give a complete classification, up to
stable isomorphism, of all $C^*$-algebras of graphs with two vertices
that have precisely one proper nontrivial ideal.  Combined with other
results, this allows us to give a complete classification of all
$C^*$-algebras of graphs satisfying Condition~(K) with exactly two
vertices.

\cbstart{}%
If $E$ is a graph with two vertices, and if $C^*(E)$ has exactly one
\cbend{}%
proper ideal, then $E$ must have exactly one proper nonempty saturated
hereditary subset with no breaking vertices.  This occurs precisely
when the vertex matrix of $E$ has the form
\[
\begin{bmatrix}a&b\\0&d\end{bmatrix}
\] where $a,d\in\{0,2,3,\dots,\infty\}$ and
$b\in\{1,2,3,\dots,\infty\}$ with the extra conditions
\[ a=0 \ \Longrightarrow \ b = \infty \qquad \text{ and } \qquad b =
\infty \implies (a= 0 \text{ or } a = \infty),
\] Computing $K$-groups using \cite{ddmt:ckegc}, we see that in all of
these cases the $K_1$-groups of $C^*(E)$, the unique proper nontrivial
ideal $I$, and the quotient $C^*(E)/I$ all vanish.  Thus the six-term
exact sequence becomes $0 \to K_0(I) \to K_0(C^*(E)) \to K_0(C^*(E)/I)
\to 0$, and using \cite{ddmt:ckegc} to compute the $K_0$-groups and
the induced maps we obtain the following cases.

$ $

\begin{center}
\begin{tabular}{|c|c|c||c|c|}\hline $a$&$d$&$b$&$\ \ K_0(I) \to
K_0(C^*(E)) \to K_0(C^*(E)/I) \ \ $&Case\\\hline\hline
0&0&$\infty$&\rule[-3mm]{0mm}{8mm}$\ZZ\to\Z \oplus \Z
\to\ZZ$&$\fifi$\\ \cline{2-5}
0&$n$&$\infty$&\rule[-3mm]{0mm}{8mm}$\Z_{d-1}\to\Z_{d-1}\oplus
\Z\to\ZZ$&$\infi$\\ \cline{2-5}
0&$\infty$&$\infty$&\rule[-3mm]{0mm}{8mm}$\ZZp\to\Z \oplus \Z
\to\ZZ$&$\infi$\\ \hline $n$&0&$1,n$&\rule[-3mm]{0mm}{8mm}$\ZZ\to
\operatorname{coker}(\smalltwobyone{b}{a-1}) \to
\Z_{a-1}$&$\fiin$\\\cline{2-5}
$n$&$n$&$1,n$&\rule[-3mm]{0mm}{8mm}$\Z_{d-1}\to
\operatorname{coker}(\smalltwobytwo{d-1}{b}{0}{a-1})\to
\Z_{a-1}$&$\inin$\\\cline{2-5}
$n$&$\infty$&$1,n$&\rule[-3mm]{0mm}{8mm}$\ZZp\to
\operatorname{coker}(\smalltwobyone{b}{a-1}) \to
\Z_{a-1}$&$\inin$\\\hline
$\infty$&0&$1,n,\infty$&\rule[-3mm]{0mm}{8mm}$\ZZ\to\Z \oplus \Z
\to\ZZp$&$\fiin$\\\cline{2-5}
$\infty$&$n$&$1,n,\infty$&\rule[-3mm]{0mm}{8mm}$\Z_{d-1}\to\Z_{d-1}\oplus\Z
\to\ZZp$&$\inin$\\\cline{2-5}
$\infty$&$\infty$&$1,n,\infty$&\rule[-3mm]{0mm}{8mm}$\ZZp\to\Z \oplus
\Z \to\ZZp$&$\inin$\\\hline
\end{tabular}
\end{center}

$ $

\noindent where ``$n$'' indicates an integer $\geq 2$, ``$\ZZ$''
indicates a copy of $\Z$ ordered with $\Z_+=\N$ and ``$\ZZp$"
indicates a copy of $\Z$ ordered with $\Z_+=\Z$.  In addition, in all
cases we have written the middle group in such a way that the map from
$K_0(I)$ to $K_0(C^*(E))$ is $[x] \mapsto [(x,0)]$, and the map from
$K_0(C^*(E))$ to $K_0(C^*(E)/I)$ is $[(x,y)] \mapsto [y]$.  Note that
in all but the first case, the order structure of the middle
$K_0$-groups is irrelevant and need not be computed.

\begin{theorem}\label{classthmtbt}\cbstart{}%
Let $E$ and $E'$ be graphs each with two vertices such that $C^*(E)$ and $C^*(E')$ each have exactly
one proper nontrivial ideal, and write the vertex matrix of $E$ as
\cbend{}%
$\smalltwobytwo{a}{b}{0}{d}$ and the vertex matrix of $E'$ as
$\smalltwobytwo{a'}{b'}{0}{d'}$. Then
\[ C^*(E)\otimes\KKK \cong C^*(E')\otimes\KKK
\] if and only if the following three conditions hold:
\begin{enumerate}
\item $a=a'$
\item $d=d'$
\item If $a\in\{2,\dots\}$ then
\begin{enumerate}
\item[(a)] If $d\in\{0,\infty\}$ then $[b] = [z] [b']$ in $\Z_{a-1}$
for a unit $[z] \in \Z_{a-1}$
\item[(b)] If $d\in\{2,\dots\}$ then $[z_1][b] = [z_2][b']$ in
$\Z_{\gcd{(a-1, d-1)}}$ for a unit $[z_1] \in \Z_{d-1}$ and a unit
$[z_2] \in \Z_{a-1}$.
\end{enumerate}
\end{enumerate}
\end{theorem}

\begin{proof} Suppose $C^*(E)\otimes\KKK \cong C^*(E')\otimes\KKK$.
Then $K_0(I) \cong K_0(I')$ as ordered groups and $K_0(C^*(E)/I) \cong
K_0(C^*(E')/I')$ as ordered groups.  From a consideration of the
invariants in the above table, this implies that $a = a'$, $d = d'$,
and the invariants for $C^*(E)$ and $C^*(E')$ both fall into the same
case (i.e.~the same row) of the table.  Thus we need only consider the
two cases described in (3)(a) and (3)(b).

$ $

\cbstart{}%
\noindent \textsc{Case i:} $a \in \{2, \ldots \}$ and $d \in \{ 0,
\infty \}$.
\cbend{}%

In this case there are isomorphisms $\alpha$, $\beta$, and $\gamma$
such that
\[ \xymatrix{ 0 \ar[r] & \Z \ar[r] \ar[d]_\alpha & \coker
(\smalltwobyone{b}{a-1}) \ar[r] \ar[d]_\beta & \Z_{a-1} \ar[r]
\ar[d]_\gamma & 0 \\ 0 \ar[r] & \Z \ar[r] & \coker
(\smalltwobyone{b'}{a-1}) \ar[r] & \Z_{a-1} \ar[r] & 0}
\] commutes.  Since the only automorphisms on $\Z$ are $\pm
\operatorname{Id}$, we have that $\alpha(x) = \pm x$.  Also, since the
only automorphisms on $\Z_{a-1}$ are multiplication by a unit, $\gamma
([x]) = [z] [x]$ for some unit $[z] \in \Z_{a-1}$.  By the
commutativity of the left square $\beta([1,0]) = [(\pm1, 0)]$.  Also,
by the commutativity of the right square, $\beta([0,1]) = ([y, z])$
for some $y \in \Z$.  It follows from the $\Z$-linearity of $\beta$
that $\beta[(r,s)] = [(\pm r + sy, sz)]$, so $\beta$ is equal to left
multiplication by the matrix $\smalltwobytwo{\pm 1}{y}{0}{z}$.  We
must have $\beta [(b, a-1)] = [(0,0)]$, and thus $[(\pm b + (a-1)y,
(a-1)z)] = [(0,0)]$ in $\coker (\smalltwobyone{b'}{a-1})$.  Hence $\pm
b + (a-1)y = b't$ and $(a-1)z = (a-1)t$ for some $t \in \Z$.  It
follows that $z = t$ and $\pm b + (a-1)y = b'z$, so $b \equiv \pm z
\mod (a-1)$.  Since $[\pm z]$ is a unit for $\Z_{a-1}$ it follows that
$[b] = [z] [b']$ in $\Z_{a-1}$ for a unit $[z] \in \Z_{a-1}$.  Thus
the condition in (a) holds.

$ $

\cbstart{}%
\noindent \textsc{Case ii:} $a \in \{2, \ldots \}$ and $d \in \{2,
\ldots \}$.
\cbend{}%
In this case there are isomorphisms $\alpha$, $\beta$,
and $\gamma$ such that
\[ \xymatrix{ 0 \ar[r] & \Z_{d-1} \ar[r] \ar[d]_\alpha & \coker
(\smalltwobytwo{d-1}{b}{0}{a-1}) \ar[r] \ar[d]_\beta & \Z_{a-1} \ar[r]
\ar[d]_\gamma & 0 \\ 0 \ar[r] & \Z_{d-1} \ar[r] & \coker
(\smalltwobytwo{d-1}{b'}{0}{a-1}) \ar[r] & \Z_{a-1} \ar[r] & 0}
\] commutes.  Since the only automorphisms on $\Z_{d-1}$ are
multiplication by a unit, we have that $\alpha([x]) = [z_1] [x]$ for
some unit $[z_1] \in \Z_{d-1}$.  Likewise, $\gamma([x]) = [z_2][x]$
for some unit $[z_2] \in \Z_{a-1}$.  By the commutativity of the left
square $\beta([1,0]) = [(z_1, 0)]$.  Also, by the commutativity of the
right square, $\beta([0,1]) = ([y, z_2])$ for some $y \in \Z$.  It
follows from the $\Z$-linearity of $\beta$ that $\beta[(r,s)] = [(z_1r
+ ys, z_2s)]$, so $\beta$ is equal to left multiplication by the
matrix $\smalltwobytwo{z_1}{y}{0}{z_2}$.  Since
$\smalltwobytwo{d-1}{b}{0}{a-1} \smalltwobyone{0}{1} =
\smalltwobyone{b}{a-1}$, we must have $\beta [(b, a-1)] = [(0,0)]$,
and thus $[(z_1 b + y(a-1), z_2(a-1))] = [(0,0)]$ in $\coker
(\smalltwobytwo{d-1}{b'}{0}{a-1})$.  Hence $z_1 b + y(a-1) = (d-1)s +
b't$ and $z_2(a-1) = (a-1)t$ for some $s, t \in \Z$.  It follows that
$z_2 = t$ and $z_1 b + y(a-1) = (d-1)s + b'z_2$.  Writing $(d-1)s -
y(a-1) = k \gcd{(a-1, d-1)}$ we obtain $z_1b - z_2b' = k \gcd{(a-1,
d-1)}$ so that $z_1 b \equiv z_2 b' \mod \gcd{(a-1, d-1)}$ and
$[z_1][b] = [z_2][b']$ in $\Z_{\gcd{(a-1, d-1)}}$.  Thus the condition
in (b) holds.

$ $

For the converse, we assume that the conditions in (1)--(3) hold.
Consider the following three cases.

$ $

\noindent \textsc{Case I:} $a = 0$ or $a = \infty$.  In this case, by
considering the invariants listed in the above table, we see that we
may use the identity maps for the three vertical isomorphisms to
obtain a commutative diagram.  Thus the six-term exact sequences are
\cbstart{}%
isomorphic, and it follows from Theorem~\ref{clas} that
\cbend{}%
$C^*(E)\otimes\KKK \cong C^*(E')\otimes\KKK$.

$ $

\noindent \textsc{Case II:} $a \in \{2, \ldots \}$ and $[b] = [z]
[b']$ in $\Z_{a-1}$ for a unit $[z] \in \Z_{a-1}$.  Then $b \cong zb'
\mod (a-1)$.  Hence $zb' - b = (a-1)y$ for some $y \in \Z$.  Consider
$\smalltwobytwo{1}{y}{0}{z} : \Z \oplus \Z \to \Z \oplus \Z$.  It is
straightforward to check that this matrix takes $\im
\smalltwobyone{b}{a-1}$ into $\im \smalltwobyone{b'}{a-1}$.  Thus
multiplication by this matrix induces a map $\beta : \coker
(\smalltwobyone{b}{a-1}) \to \coker (\smalltwobyone{b'}{a-1})$.  In
addition, if we let $\alpha = \operatorname{Id}$ and let $\gamma$ be
multiplication by $[z]$, then it is straightforward to verify that the
diagram
\[ \xymatrix{ 0 \ar[r] & \Z \ar[r] \ar[d]_\alpha & \coker
(\smalltwobyone{b}{a-1}) \ar[r] \ar[d]_\beta & \Z_{a-1} \ar[r]
\ar[d]_\gamma & 0 \\ 0 \ar[r] & \Z \ar[r] & \coker
(\smalltwobyone{b'}{a-1}) \ar[r] & \Z_{a-1} \ar[r] & 0}
\] commutes.  Since $\alpha$ and $\gamma$ are isomorphisms, an
application of the five lemma implies that $\beta$ is an isomorphism.
It follows from Theorem~\ref{clas} that $C^*(E)\otimes\KKK \cong
C^*(E')\otimes\KKK$.

$ $

\noindent \textsc{Case III:} Suppose that $[z_1][b] = [z_2][b']$ in
$\Z_{\gcd{(a-1, d-1)}}$ for a unit $[z_1] \in \Z_{d-1}$ and a unit
$[z_2] \in \Z_{a-1}$.  Then $z_1b - z_2b' = k \gcd{(a-1, d-1)}$ for
some $k \in \Z$.  Furthermore, we may write $k \gcd{(a-1, d-1)} =
s(d-1) -y(a-1)$ for some $s,y \in \Z$.  Consider
$\smalltwobytwo{z_1}{y}{0}{z_2} : \Z \oplus \Z \to \Z \oplus \Z$.  It
is straightforward to check that this matrix takes $\im
\smalltwobytwo{d-1}{b}{0}{a-1}$ into $\im
\smalltwobytwo{d-1}{b'}{0}{a-1}$.  Thus multiplication by this matrix
induces a map $\beta : \coker (\smalltwobytwo{d-1}{b}{0}{a-1}) \to
\coker (\smalltwobyone{d-1}{b'}{0}{a-1})$.  In addition, if we let
$\alpha$ be multiplication by $[z_1]$ and and let $\gamma$ be
multiplication by $[z_2]$, then it is straightforward to verify that
the diagram
\[ \xymatrix{ 0 \ar[r] & \Z \ar[r] \ar[d]_\alpha & \coker
(\smalltwobytwo{d-1}{b}{0}{a-1}) \ar[r] \ar[d]_\beta & \Z_{a-1} \ar[r]
\ar[d]_\gamma & 0 \\ 0 \ar[r] & \Z \ar[r] & \coker
(\smalltwobytwo{d-1}{b'}{0}{a-1}) \ar[r] & \Z_{a-1} \ar[r] & 0}
\] commutes.  Since $\alpha$ and $\gamma$ are isomorphisms, an
application of the five lemma implies that $\beta$ is an isomorphism.
It follows from Theorem~\ref{clas} that $C^*(E)\otimes\KKK \cong
C^*(E')\otimes\KKK$.
\end{proof}

\begin{example}\label{fourex} Consider the three graphs

$ $

\[ \xymatrix{ \ar@(u, ur)[] \ar@(u, ul)[] \ar@(d, dr)[] \ar@(d, dl)[]
\bullet \ar[rr] & & \bullet} \qquad \qquad \xymatrix{\ar@(u, ur)[]
\ar@(u, ul)[] \ar@(d, dr)[] \ar@(d, dl)[] \bullet \ar@/^/[rr]
\ar@/_/[rr] & & \bullet} \qquad \qquad \xymatrix{\ar@(u, ur)[] \ar@(u,
ul)[] \ar@(d, dr)[] \ar@(d, dl)[] \bullet \ar[rr] \ar@/^0.9pc/[rr]
\ar@/_0.9pc/[rr] & & \bullet }
\]

$ $

$ $

\noindent which all have graph $C^*$-algebras with precisely one
proper nontrivial ideal. By Theorem~\ref{classthmtbt} the
$C^*$-algebras of the two first graphs are stably isomorphic to each
other, but not to the $C^*$-algebra of the third graph. 
\end{example}

\begin{remark}
We mention that with existing technology it would be very difficult to see directly that the $C^*$-algebras of the two first graphs in Example~\ref{fourex} are stably isomorphic.  One approach would be to form the \emph{stabilized graphs} (see \cite[\S4]{tom:socatg}) and then attempt to transform one graph to the other through operations that preserve the stable isomorphism class of the associated $C^*$-algebra (e.g., in/outsplittings, delays).  However, even in this concrete example it is unclear what sequence of operations would accomplish this and we speculate that it would not be possible using the types of operations mentioned above and their inverses.  In addition, the second author's Ph.D. thesis (see \cite{Tom-thesis} and the resulting papers \cite{rtw:ctcgws}, \cite{tom:ecefcgws}, and \cite{tom:cefga}) deals with extensions of graph $C^*$-algebras and shows that under certain circumstances two essential one-sink extensions of a fixed graph $G$ have stably isomorphic $C^*$-algebras if they determine the same class in $\operatorname{Ext} C^*(G)$ \cite[Theorem~4.1]{tom:ecefcgws}.  In Example~\ref{fourex}, the three displayed graphs are all essential one-sink extensions of the graph with one vertex and four edges, whose $C^*$-algebra is $\mathcal{O}_4$.  We also have that $\operatorname{Ext} \mathcal{O}_4 \cong \Z_3$, and the first two graphs in Example~\ref{fourex} determine the classes $[1]$ and $[2]$ in $\Z_3$, respectively.  Consequently, we cannot apply \cite[Theorem~4.1]{tom:ecefcgws}, and we see that the methods of this paper have applications to situations not covered by \cite[Theorem~4.1]{tom:ecefcgws}. (As an aside, we mention that the second author has conjectured that if $G$ is a finite graph with no sinks or sources, if $C^*(G)$ is simple, and if $E_1$ and $E_2$ are one-sink extensions of $G$, then $C^*(E_1)$ is stably isomorphic to $C^*(E_2)$ if and only if there exists an automorphism on $\operatorname{Ext} C^*(G)$ taking the class of the extension determined by $E_1$ to the class of the extension determined by $E_2$.  We see that Example~\ref{fourex} is consistent with this conjecture since there is an automorphism of $\Z_3$ taking $[1]$ to $[2]$.)
\end{remark}

$ $

Using the Kirchberg-Phillips Classification Theorem and our results in
Theorem~\ref{classthmtbt} we are able to give a complete
classification of the stable isomorphism classes of $C^*$-algebras of
graphs satisfying Condition~(K) with exactly two vertices.  We state
\cbstart{}%
this result in the following theorem.  As one can see, there are a
\cbend{}%
variety of cases and possible ideal structures for these stable
isomorphism classes.

\begin{theorem}\label{class-two-vertices} Let $E$ and $E'$ be graphs
satisfying Condition~(K) that each have exactly two vertices.  Let
$A_E$ and $A_{E'}$ be the vertex matrices of $E$ and $E'$,
respectively, and order the vertices of each so that $c \leq b$ and
$c' \leq b'$.  Then $C^*(E) \otimes \KKK \cong C^*(E') \otimes \KKK$
if and only if one of the following five cases occurs.
\begin{itemize}
\item[(i)] $A_E = \smalltwobytwo{a}{b}{c}{d}$ and $A_E =
\smalltwobytwo{a'}{b'}{c'}{d'}$ with
$$\quad (b \neq 0 \text{ and } c \neq 0) \quad \text{ or } \quad (a = 0, 0 < b < \infty, c= 0, \text{and } d \geq 2)$$
and
$$\qquad (b' \neq 0 \text{ and } c' \neq 0) \quad  \text{ or } \quad (a' = 0, 0 < b' < \infty, c'= 0, \text{and } d' \geq 2)$$
and if $B_E$ is the $E^0 \times E^0_\textnormal{reg}$ submatrix of
$A_E^t-I$ and $B_{E'}$ is the $(E')^0 \times (E')^0_\textnormal{reg}$
submatrix of $A_{E'}^t-I$, then $$\coker (B_E :
\Z^{E^0_\textnormal{reg}} \to Z^{E^0}) \cong \coker (B_{E'} :
\Z^{(E')^0_\textnormal{reg}} \to Z^{(E')^0})$$ and $$\ker (B_E :
\Z^{E^0_\textnormal{reg}} \to Z^{E^0}) \cong \ker (B_{E'} :
\Z^{(E')^0_\textnormal{reg}} \to Z^{(E')^0}).$$ In this case $C^*(E)$
and $C^*(E')$ are purely infinite and simple.

\item[(ii)] $A_E = \smalltwobytwo{0}{b}{0}{0}$ and $A_{E'} =
\smalltwobytwo{0}{b'}{0}{0}$ with $0 < b < \infty$ and $0 < b' <
\infty$.  In this case $C^*(E) \cong M_{b+1}(\C)$ and $C^*(E') \cong
M_{b'+1}(\C)$, so that both $C^*$-algebras are simple and
finite-dimensional.

\item[(iii)] $A_E = \smalltwobytwo{a}{b}{0}{d}$ and $A_{E'} =
\smalltwobytwo{a'}{b'}{0}{d'}$ with $b \neq 0$ and $b' \neq 0$,
\begin{center} $a=0 \ \Longrightarrow \ b = \infty \qquad \text{ and }
\qquad b = \infty \implies (a= 0 \text{ or } a = \infty),$ \end{center}
and
\begin{center} $a'=0 \ \Longrightarrow \ b' = \infty \qquad \text{ and }
\qquad b' = \infty \implies (a'= 0 \text{ or } a' = \infty),$ \end{center}
and the conditions (1)--(3) of Theorem~\ref{classthmtbt} hold.
In this case $C^*(E)$ and $C^*(E')$ each have exactly
one proper nontrivial ideal and have ideal structure of the form
\[ \xymatrix{ A \ar@{-}[d] \\ \ar@{-}[d] I \\ \ \{ 0 \}. }
\]

\item[(iv)] $A_E = \smalltwobytwo{a}{\infty}{0}{d}$ and $A_{E'} =
\smalltwobytwo{a'}{\infty}{0}{d'}$ with $a \in \{ 2, 3, \ldots \}$ and $a' \in \{2, 3, \ldots \}$,
and with $a = a'$ and $d = d'$.  In this case $C^*(E)$ and $C^*(E')$
each have exactly two proper nontrivial ideals and have ideal
structure of the form
\[ \xymatrix{ A \ar@{-}[d] \\ \ar@{-}[d] I \\ \ar@{-}[d] J \\ \ \{ 0
\}. }
\]

\item[(v)] $A_E = \smalltwobytwo{a}{0}{0}{d}$ and $A_{E'} =
\smalltwobytwo{a'}{0}{0}{d'}$ with $$(a = a' \text{ and } d = d')
\quad \text{ or } \quad (a = d' \text{ and } d = a').$$ In this case
$C^*(E) \cong C^*(E') \cong I \oplus J$, where $I := \begin{cases}
\mathcal{O}_a & \text{ if $a \geq 2$} \\ \C & \text{ if $a
=0$} \end{cases}$ and $J := \begin{cases} \mathcal{O}_d & \text{ if $d
\geq 2$} \\ \C & \text{ if $d =0$} \end{cases}$, and each
$C^*$-algebra has exactly two proper nontrivial ideals and ideal
structure of the form
\[ \xymatrix{ & A \ar@{-}[dr] \ar@{-}[dl] & \\ I \ar@{-}[dr] & & J
\ar@{-}[dl] \\ & \ \{ 0 \}. & }
\]
\end{itemize}
\end{theorem}

\begin{remark} We are not able to classify $C^*$-algebras of graphs
with exactly two vertices that do not satisfy Condition~(K).  For
example if $E$ and $E'$ are graphs with vertex matrices $A_E =
\smalltwobytwo{1}{b}{0}{1}$ and $A_{E'} =
\smalltwobytwo{1}{b'}{0}{1}$, then $C^*(E)$ and $C^*(E')$ each have
uncountably many ideals, and are extensions of $C(\T)$ by $C(\T
\otimes \KKK)$.  Using existing techniques, it is unclear when
$C^*(E)$ and $C^*(E')$ will be stably isomorphic.
\end{remark}

We conclude this section with an example showing an application of
Theorem~\ref{clasimp} to $C^*$-algebras with multiple proper ideals.

\begin{example} Consider the two graphs
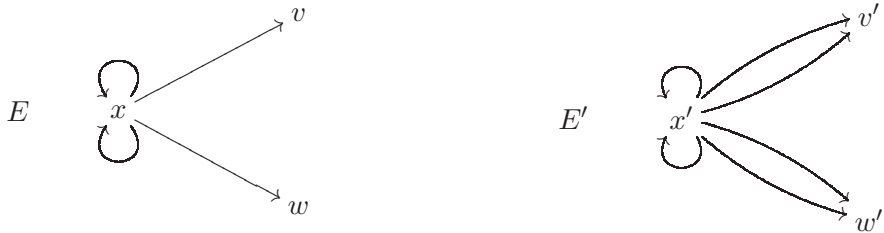
\[ \qquad \xymatrix{ & & & v \\ E & \ar@(ur, ul)[] \ar@(dr, dl)[]
\ar[rru] \ar[rrd] x & & \\ & & & w } \qquad \qquad \qquad
\qquad\xymatrix{ & & & v' \\ E' & \ar@(ur, ul)[] \ar@(dr, dl)[]
\ar@/^/[rru] \ar@/_/[rru] \ar@/^/[rrd] \ar@/_/[rrd] x' & & \\ & & & w'
}
\] The ideal $I := I_{\{v,w\}}$ in $C^*(E)$ is a largest proper ideal
that is an AF-algebra, and the six-term exact sequence corresponding
to $$0 \to I \to C^*(E) \to C^*(E)/I \to 0$$ is $$0 \to \Z \oplus \Z
\to \coker ( \left[ \begin{smallmatrix} 1 \\ 1 \\ 1 \end{smallmatrix}
\right] ) \to 0$$ where the middle map is $[(x,y)] \mapsto [(x,y,0)]$.
Likewise, the ideal $I' := I_{\{v',w'\}}$ in $C^*(E')$ is a largest
proper ideal that is an AF-algebra, and the six-term exact sequence
corresponding to $$0 \to I' \to C^*(E') \to C^*(E')/I' \to 0$$ is $$0
\to \Z \oplus \Z \to \coker ( \left[ \begin{smallmatrix} 2 \\ 2 \\
1 \end{smallmatrix} \right] ) \to 0$$ where the middle map is $[(x,y)]
\mapsto [(x,y,0)]$.  If we define $\beta : \coker (
\left[ \begin{smallmatrix} 1 \\ 1 \\ 1 \end{smallmatrix} \right] ) \to
\coker ( \left[ \begin{smallmatrix} 2 \\ 2 \\ 1 \end{smallmatrix}
\right] )$ by $\beta[(x,y,z)] = [(x+z, y+z, z)]$, then we see that the
diagram
\[ \xymatrix{ 0 \ar[r] & \Z \oplus \Z \ar[d]_{\operatorname{Id}}
\ar[r] & \coker{(\left[\begin{smallmatrix}1\\1
\\1\end{smallmatrix}\right])} \ar[d]_{\beta} \ar[r] & 0 \\ 0 \ar[r] &
\Z \oplus \Z \ar[r] &\coker{(\left[\begin{smallmatrix}2\\2
\\1\end{smallmatrix}\right])} \ar[r] & 0 }
\] commutes.  An application of the five lemma shows that $\beta$ is
an isomorphism.  It follows from Theorem~\ref{clasimp} that $C^*(E)
\otimes \KKK \cong C^*(E') \otimes \KKK$.

\end{example}

In the examples above, both connecting maps in the six-term exact sequences vanish.  Since all $C^*$-algebras considered (and, more generally, all graph $C^*$-algebras satisfying Condition~(K)) have real rank zero, the exponential map $\partial : K_0(A/I) \to K_1(I)$ is always zero.  However, the index map $\partial : K_1(A/I) \to K_0(I)$ does not necessarily vanish and may carry important information.  In forthcoming work, the authors and Carlsen explain how to compute this map for graph $C^*$-algebras.

\section{Stability of ideals} \label{stab-ideals-sec}

In this section we prove that if $A$ is a graph $C^*$-algebra that is
not an AF-algebra, and if $A$ contains a unique proper nontrivial
ideal $I$ , then $I$ is stable.

\begin{definition} If $v$ is a vertex in a graph $E$ we define $$L(v)
:= \{w \in E^0 : \text{ there is a path from $w$ to $v$} \}.$$ We say
that $v$ is \emph{left infinite} if $L(v)$ contains infinitely many
elements.
\end{definition}

\begin{definition} If $E = (E^0, E^1, r, s )$ is a graph, then a
\emph{graph trace} on $E$ is a function $g : E^0 \rightarrow
[0,\infty)$ with the following two properties:
\begin{enumerate}
\item \label{g-t-1} For any $v \in G^0$ with $0 < |s^{-1}(v)| <
\infty$ we have $g(v) = \sum_{s(e) = v} g(r(e))$.
\item \label{g-t-2} For any infinite emitter $v \in G^0$ and any
finite set of edges $e_1, \ldots, e_n \in s^{-1}(v)$ we have $g(v)
\geq \sum_{i=1}^n g(r(e_i))$.
\end{enumerate} We define the \emph{norm} of a graph trace $g$ to be
the (possibly infinite) quantity $\| g \| := \sum_{v \in E^0} g(v)$,
and we say a graph trace $g$ is \emph{bounded} if $\| g \| < \infty$.
\end{definition}

\begin{lemma} \label{stable-l-i-lem} Let $E$ be a graph such that
$C^*(E)$ is simple.  If there exists $v \in E^0$ such that $v$ is left
infinite, then $C^*(E)$ is stable.
\end{lemma}

\begin{proof} Since $C^*(E)$ is simple, it follows from
\cite[Corollary~2.15]{ddmt:cag} that $E$ is cofinal.  Therefore, the
vertex $v$ can reach every cycle in $E$, and any vertex that is on a
cycle in $E$ is left infinite.  In addition, if $g : E^0 \to
[0,\infty)$ is a bounded graph trace on $E$, then since $v$ is left
infinite, it follows that $g(v) = 0$.  Furthermore, it follows from
\cite[Lemma~3.7]{mt:sgcg} that $$H := \{ w \in E^0 : g(w) = 0 \}$$ is
a saturated hereditary subset of vertices.  Since $C^*(E)$ is simple,
it follows from \cite[Theorem~3.5]{ddmt:cag} that the only saturated
hereditary subsets of $E$ are $E^0$ and $\emptyset$.  Because $v \in
H$, we have that $H \neq \emptyset$ and hence $H = E^0$, which implies
that $g \equiv 0$.  Since we have shown that every vertex on a cycle
in $E$ is left infinite, and that there are no nonzero bounded graph
traces on $E$, it follows from \cite[Theorem~3.2(d)]{mt:sgcg} that
$C^*(E)$ is stable.
\end{proof}

\begin{proposition} \label{stable-or-AF-ideal} Let $E$ be a graph such
that $C^*(E)$ contains a unique proper nontrivial ideal $I$, and let
$\{ E^0, H, \emptyset \}$ be the saturated hereditary subsets of $E$.
Then there are two possibilities:
\begin{enumerate}
\item[(1)] The ideal $I$ is stable; or
\item[(2)] The graph $C^*$-algebra $C^*(E)$ is a nonunital AF-algebra,
and $H$ is infinite.
\end{enumerate}
\end{proposition}

\begin{proof} By Lemma~\ref{E-structure-lem}, we see that $E$ contains
a unique saturated hereditary subset $H$ not equal to either $E^0$ or
$\emptyset$, and also $I = I_H$.  In addition, it follows from
\cite[Lemma~1.6]{kdjhhsw:srga} that $I_H$ is isomorphic to the graph
$C^*$-algebra $C^*({}_HE_\emptyset)$, where ${}_HE_\emptyset$ is the
graph described in \cite[Definition~1.4]{kdjhhsw:srga}.  In
particular, if we let
$$F_H := \{ \alpha \in E^* : s(\alpha) \notin H, r(\alpha) \in H, \text{ and } r(\alpha_i) \notin H \text{ for $i < | \alpha |$} \}$$
\noindent then
$${}_HE_\emptyset^0 := H \cup F_H \quad \text{ and } \quad {}_HE_\emptyset^1 :=  \{e \in E^1 : s(e) \in H \} \cup \{ \overline{\alpha} : \alpha \in F_H \}$$
where $s(\overline{\alpha}) = \alpha$, $r(\overline{\alpha}) =
r(\alpha)$, and the range and source of the other edges is the same as
in $E$.  Note that since $I$ is the unique proper nontrivial ideal in
$C^*(E)$, we have that $I \cong C^*({}_HE_\emptyset)$ is simple.

Consider three cases.

\noindent \textsc{Case I:} $H$ is finite.

Choose a vertex $v \in E^0 \setminus H$.  By
Lemma~\ref{E-structure-lem} $v$ is not a sink in $E$, and thus there
exists an edges $e_1 \in E^1$ with $s(e_1) = v$ and $r(e_1) \notin H$.
Continuing inductively, we may produce an infinite path $e_1e_2e_3
\ldots$ with $r(e_i) \notin H$ for all $i$.  (Note that the vertices
of this infinite path need not be distinct.)  We shall show that for
each $i$ there is a path from $r(e_i)$ to a vertex in $H$.  Fix $i$,
and let
$$X := \{ w \in E^0 : \text{ there is a path from $r(e_i)$ to $w$} \}.$$  Then $X$ is a nonempty hereditary subset, and by Lemma~\ref{E-structure-lem} it follows that $X \cap H \neq \emptyset$.  Thus there is a path from $r(e_i)$ to a vertex in $H$.  Since this is true for all $i$, it must be the case that $F_H$ is infinite.  In the graph ${}_HE_\emptyset$ there is an edge from each element of $F_H$ to an element in $H$.  Since $H$ is finite, this implies that there is a vertex in $H \subseteq {}_HE_\emptyset^0$ that is reached by infinitely many vertices, and hence is left infinite.  It follows from Lemma~\ref{stable-l-i-lem} that $I \cong C^*({}_HE_\emptyset)$ is stable.  Thus we are in the situation described in (1).

\noindent \textsc{Case II:} $H$ is infinite, and $E$ contains a cycle.

Let $\alpha = \alpha_1\ldots \alpha_n$ be a cycle in $E$.  Since $H$
is hereditary, the vertices of $\alpha$ must either all lie outside of
$H$ or all lie inside of $H$.  If the vertices all lie in $H$, then
the graph ${}_HE_\emptyset$ contains a cycle, and since
$C^*({}_HE_\emptyset)$ is simple, the dichotomy for simple graph
$C^*$-algebras \cite[Remark~2.16]{ddmt:cag} implies that
$C^*({}_HE_\emptyset)$ is purely infinite.  Since $H$ is infinite, it
follows that ${}_HE_\emptyset^0$ is infinite and
$C^*({}_HE_\emptyset)$ is nonunital.  Because $C^*({}_HE_\emptyset)$
is a simple, separable, purely infinite, and nonunital $C^*$-algebra,
Zhang's Theorem \cite{sz:dpmamc} implies that $I \cong
C^*({}_HE_\emptyset)$ is stable.  Thus we are in the situation
described in (1).

If the vertices of $\alpha$ all lie outside $H$, then the set $$X :=
\{ w \in E^0 : \text{ there is a path from $r(\alpha_n)$ to $w$} \}$$
is a nonempty hereditary set.  It follows from
Lemma~\ref{E-structure-lem} that $X \cap H \neq \emptyset$.  Thus
there exists a vertex $v \in H$ and a path $\beta$ from $r(\alpha_n)$
to $v$ with $r(\beta_i) \notin H$ for $i < | \beta |$.  Consequently
there are infinitely many paths in $F_H$ that end at $v$ (viz.~
$\beta, \alpha\beta, \alpha \alpha \beta, \alpha\alpha\alpha\beta,
\ldots$).  Hence there are infinitely many vertices in
${}_HE_\emptyset$ that can reach $v$, and $v$ is a left infinite
vertex in ${}_HE_\emptyset$.  It follows from
Lemma~\ref{stable-l-i-lem} that $I \cong C^*({}_HE_\emptyset)$ is
stable.  Thus we are in the situation described in (1).

\noindent \textsc{Case III:} $H$ is infinite, and $E$ does not contain
a cycle.

Since $E$ does not contain a cycle, it follows from
\cite[Corollary~2.13]{ddmt:cag} that $C^*(E)$ is an AF-algebra.  In
addition, since $H$ is infinite it follows that $E^0$ is infinite and
$C^*(E)$ is nonunital.  Thus we are in the situation described in (2).
\end{proof}

\begin{corollary}\label{usenext} If $E$ is a graph with a finite
number of vertices and such that $C^*(E)$ contains a unique proper
nontrivial ideal $I$, then $I$ is stable.  Furthermore, if $\{E^0, H,
\emptyset \}$ are the saturated hereditary subsets of $E$, then
$C^*(E_H)$ is a unital $C^*$-algebra and $I \cong C^*(E_H) \otimes
\mathcal{K}$.
\end{corollary}

\begin{proof} Since $E^0$ is finite it is the case that $C^*(E)$ is
unital, and it follows from Proposition~\ref{stable-or-AF-ideal} that
$I$ is stable.  Furthermore, since $I = I_H$ it follows from
\cite[Theorem~4.1]{tbdpirws:crg} and
\cite[Proposition~3.4]{tbdpirws:crg} that $I$ is Morita equivalent to
$C^*(E_H)$.  Since $I$ and $C^*(E_H)$ are separable, it follows that
$I$ and $C^*(E_H)$ are stably isomorphic.  Thus $I \cong I \otimes
\mathcal{K} \cong C^*(E_H) \otimes \mathcal{K}$.  Finally, since
$E_H^0 = H \subseteq E^0$ is finite, $C^*(E_H)$ is unital.
\end{proof}

%

\newcommand{\arxiv}[2]{\href{#2}{#1}}

\bibliographystyle{amsalpha}

\end{document}